\documentclass[a4paper]{scrartcl}
\usepackage{amsmath}
\usepackage{amssymb}
\usepackage[T1]{fontenc}
\usepackage{amsthm}
\usepackage{longtable}
\usepackage{stmaryrd}
\usepackage{color}
\usepackage{hyperref}
\usepackage[ansinew]{inputenc}
\usepackage{graphicx,subfigure}
\usepackage{tikz}

\begin{document}
\theoremstyle{remark}
\newtheorem{theorem}{\textbf{Theorem}}
\newtheorem{lemma}[theorem]{\textbf{Lemma}}
\newtheorem{corollary}[theorem]{\textbf{Corollary}}
\newtheorem{proposition}[theorem]{\textbf{Proposition}}
\newtheorem{example}{\textbf{Example}}
\newtheorem{claim}{\textbf{Claim}}
\newtheorem{algorithm}{\textbf{Algorithm}}
\newtheorem{definition}{\textbf{Definition}}
\newtheorem{case}{\textbf{Case}}
\newtheorem*{beweis}{\textbf{Proof:}}
\newtheorem{observation}{\textbf{Observation}}

\begin{center}
\large{\textsc{The inducibility of small oriented graphs}}
\end{center}
\begin{center}
\large{\textsc{Konrad Sperfeld}}\linebreak
Universität Rostock, Institut für Mathematik\linebreak
D-18057 Rostock, Germany \linebreak
\textit{Konrad.Sperfeld@uni-rostock.de}
\end{center}
\begin{abstract}
Abstract: We use Razborov's flag algebra method \cite{flagalgebra} to show an asymptotic upper bound for the maximal induced density $i(\vec P_3)$ of the orgraph $\vec P_3$ in an arbitrary orgraph. A conjecture of Thomassé states that $i(\vec P_3)=\frac{2}{5}$. The hitherto best known upper bound $i(\vec P_3)\leq\frac{12}{25}$ was given by Bondy. We can show that $i(\vec P_3)\leq 0.4446$. Further, we consider such a maximal density for some other small orgraphs. With easy arguments one can see that $i(\vec C_3)=\frac{1}{4}$, $i(\vec K_2 \cup \vec E_1)=\frac{3}{4}$ and $\frac{2}{21}\leq i(\vec C_4)$. We show that $i(\vec C_4)\leq 0.1104$ and conjecture that the extremal orgraphs of $\vec P_3$ and $\vec C_4$ are the same. Furthermore we show that $6-4\sqrt{2}\leq i(\vec K_{1,2})\leq 0.4644$.
\end{abstract}
\section{Introduction}
The whole paper deals with oriented graphs (we will call them orgraphs), thus graphs with directed edges, no loops, no bidirected edges and no multiple edges. For an orgraph $\Gamma=(V,E)$ we write $V_\Gamma$ for its set of vertices and $\left|\Gamma\right|$ for its number of vertices as well as $E_\Gamma$ for its set edges and $\left\|\Gamma\right\|$ for its number of edges. The most famous and surely most studied problem on orgraphs is the Cacetta-Häggvist conjecture \cite{CH}, which was made in 1978. In this paper we want to turn our attention to another problem in this area. For an orgraph $\Gamma$ let $max(\Gamma;n)$ denote the maximal number of sets $T\subseteq V_\Gamma$ with $\left|T\right|=\left|\Gamma\right|$ such that $T$ induces a subgraph which is isomorphic to $\Gamma$ in an $n$-vertex orgraph. Remark, that we don't count possible symmetries of $\Gamma$. Now the inducibility of an orgraph $\Gamma$ is defined as
\begin{equation*}
i(\Gamma):=\limsup\limits_{n\rightarrow\infty}\frac{max(\Gamma;n)}{\binom{n}{\left|\Gamma\right|}},
\end{equation*}
the asymptotic value of the maximal density of $\Gamma$ in any orgraph. Thus, in an arbitrary orgraph the maximal density of orgraphs $\Gamma$ is $i(\Gamma)+o(1)$.\newline
There are several papers (see \cite{bollobas}, \cite{brown}, \cite{exoo}, \cite{hatami} and \cite{hirst}), where the inducibility on undirected simple graphs is investigated. There are some small orgraphs, whose inducibility is not known yet. On these we want to focus now.\newline
For two orgraphs $\Gamma_1=(V_1,E_1)$ and $\Gamma_2=(V_2,E_2)$ the lexicographic graph product $\circ$ is defined as the following orgraph.
\begin{equation*}
\Gamma_1\circ \Gamma_2:= \left(V_1\times V_{2},\left\{(d_1,d_2)(d_1',d_2') \, | \; d_1d_1'\in E_{1} \text{ or } \left(d_1=d_1' \text{ and } d_2d_2'\in E_{2}\right)\right\}\right)
\end{equation*}
Furthermore, we define
\begin{equation*}
\Gamma^{\circ n}:=\underbrace{\Gamma\circ \Gamma \circ \ldots \circ \Gamma}_{n\text{-times}}.
\end{equation*}
Thus, the lexicographic product $\Gamma_1\circ \Gamma_2$ is a copy of the orgraph $\Gamma_1$, where each vertex of $\Gamma_1$ is replaced by a copy of $\Gamma_2$. 
\newline
Figure \ref{fig:graph} defines some orgraphs, which we will need in the following.\newline
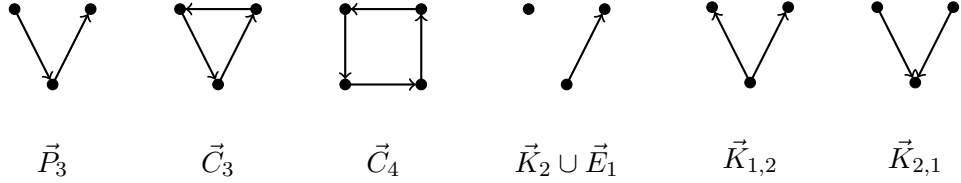
\begin{figure}[ht]
\begin{center}
\begin{tikzpicture}
\filldraw[fill=black] (1,1) circle (0.07cm);
\filldraw[fill=black] (1.5,2) circle (0.07cm);
\filldraw[fill=black] (0.5,2) circle (0.07cm);
\draw[-to,shorten >=2pt,thick] (0.5,2)--(1,1);
\draw[-to,shorten >=2pt,thick] (1,1)--(1.5,2);
\draw  (1,0) node {$\vec P_3$};
\end{tikzpicture}
$\qquad$
\begin{tikzpicture}
\filldraw[fill=black] (1,1) circle (0.07cm);
\filldraw[fill=black] (1.5,2) circle (0.07cm);
\filldraw[fill=black] (0.5,2) circle (0.07cm);
\draw[-to,shorten >=2pt,thick] (0.5,2)--(1,1);
\draw[-to,shorten >=2pt,thick] (1,1)--(1.5,2);
\draw[-to,shorten >=2pt,thick] (1.5,2)--(0.5,2);
\draw  (1,0) node {$\vec C_3$};
\end{tikzpicture}
$\qquad$
\begin{tikzpicture}
\filldraw[fill=black] (1,1) circle (0.07cm);
\filldraw[fill=black] (1,2) circle (0.07cm);
\filldraw[fill=black] (2,1) circle (0.07cm);
\filldraw[fill=black] (2,2) circle (0.07cm);
\draw[-to,shorten >=2pt,thick] (1,2)--(1,1);
\draw[-to,shorten >=2pt,thick] (1,1)--(2,1);
\draw[-to,shorten >=2pt,thick] (2,1)--(2,2);
\draw[-to,shorten >=2pt,thick] (2,2)--(1,2);
\draw (1.5,0) node {$\vec C_4$};
\end{tikzpicture}
$\qquad$
\begin{tikzpicture}
\filldraw[fill=black] (1,1) circle (0.07cm);
\filldraw[fill=black] (1.5,2) circle (0.07cm);
\filldraw[fill=black] (0.5,2) circle (0.07cm);
\draw[-to,shorten >=2pt,thick] (1,1)--(1.5,2);
\draw  (1,0) node {$\vec K_2 \cup \vec E_1$};
\end{tikzpicture}
$\qquad$
\begin{tikzpicture}
\filldraw[fill=black] (1,1) circle (0.07cm);
\filldraw[fill=black] (1.5,2) circle (0.07cm);
\filldraw[fill=black] (0.5,2) circle (0.07cm);
\draw[-to,shorten >=2pt,thick] (1,1)--(1.5,2);
\draw[-to,shorten >=2pt,thick] (1,1)--(0.5,2);
\draw  (1,0) node {$\vec K_{1,2}$};
\end{tikzpicture}
$\qquad$
\begin{tikzpicture}
\filldraw[fill=black] (1,1) circle (0.07cm);
\filldraw[fill=black] (1.5,2) circle (0.07cm);
\filldraw[fill=black] (0.5,2) circle (0.07cm);
\draw[-to,shorten >=2pt,thick] (1.5,2)--(1,1);
\draw[-to,shorten >=2pt,thick] (0.5,2)--(1,1);
\draw  (1,0) node {$\vec K_{2,1}$};
\end{tikzpicture}
\end{center}
\caption{The definition of the orgraphs $\vec P_3$, $\vec C_3$, $\vec C_4$, $\vec K_2 \cup \vec E_1$, $\vec K_{1,2}$ and $\vec K_{2,1}$.}\label{fig:graph}
\end{figure}  
\begin{observation}$\quad$\newline
\begin{center}
$i(\vec P_3)\geq \frac{2}{5}$ and $i(\vec C_4)\geq \frac{2}{21}$.
\end{center}
\end{observation}
\begin{proof}
We have a look on the limit orgraph $\lim\limits_{n\to\infty}\left(\vec C_4\right)^{\circ n}$. Let $x$ be the density of $\vec P_3$ and $y$ the density of $\vec C_4$ in $\lim\limits_{n\to\infty}\left(\vec C_4\right)^{\circ n}$. Then 
\begin{eqnarray*}
x&=&1\cdot\frac{3}{4}\cdot\frac{2}{4}+1\cdot\frac{1}{4}\cdot\frac{1}{4}\cdot x,\\  
y&=&1\cdot\frac{3}{4}\cdot\frac{2}{4}\cdot\frac{1}{4}+1\cdot\frac{1}{4}\cdot\frac{1}{4}\cdot\frac{1}{4}\cdot y.
\end{eqnarray*} 
In both equations the first summand is the density of $\vec P_3$ (resp. $\vec C_4$), where each vertex of $\vec P_3$ (resp. $\vec C_4$) is from a different part. The second summand is the density that each vertex is from the same part. Thus, $x=\frac{2}{5}$ and $y=\frac{2}{21}$.
\end{proof}
In \cite{summaryCH} a conjecture of Thomassé can be found which claims that $i(\vec P_3)= \frac{2}{5}$. We conjecture that the extremal graphs are the same for $\vec P_3$ and $\vec C_4$. Maybe, the reason for this could be that every induced subgraph of $\vec C_4$ on $3$ vertices is isomorphic to $\vec P_3$. The best known upper bound for $i(\vec P_3)$  states that $i(\vec P_3)\leq \frac{12}{25}$ and was given by Bondy. This upper bound can be found in \cite{summaryCH} too. In section \ref{sec:mainresult} we will prove that $i(\vec P_3)\leq 0.4446$, $i(\vec C_4)\leq 0.1104$ and $6-4\sqrt{2}\leq i(\vec K_{1,2})=i(\vec K_{2,1})\leq 0.4644$. To do this we need some parts of Razborov's flag algebra method \cite{flagalgebra}. In section~\ref{sec:flag} we will roughly explain everything we need from it and show as an example for the application of the described method that $i(\vec P_3)\leq \frac{4}{7}$, $i(\vec C_3)=\frac{1}{4}$ and $i(\vec K_2 \cup \vec E_1)=\frac{3}{4}$.\newline
Let $\left[k\right]:=\left\{1,2,\ldots,k\right\}$. We write vectors underlined, e.g. $\underline{v}=\left(\underline{v}(1),\underline{v}(2),\underline{v}(3)\right)$ is a vector with three coordinates. A collection $V_1,\ldots,V_t$ of finite sets is a sunflower with center $C$ if $V_i\cap V_j=C$ for every two distinct $i,j\in\left[t\right]$.
\section{Flag Algebras}
\label{sec:flag} 
With his theory of flag algebras, Razborov developed a very strong tool for solving some classes of problems in extremal graph theory. For our proof we will just need a small part of his method, which can be thought of as an application of the Cauchy-Schwarz inequality in the theory of orgraphs. For a detailed study of flag algebras we refer the reader to Razborov's original paper \cite{flagalgebra}. In this section we will just define the most important ingredients for our calculation. Furthermore, we will give a short introduction into flag algebras.\newline
Let $\mathcal{O}$ be the family of all unlabeled orgraphs considered up to isomorphism. By $\mathcal{O}_\ell$ we denote the set of all $\Gamma\in\mathcal{O}$ with order $\ell$. A type $\sigma$ of order $k$ is a labeled orgraph of order $k$. Thus, each vertex of a type can be uniquely identified by its label. Usually, we use the elements of $\left[k\right]$ as labels.\newline
One denotes by $0$ the unique type of order $0$. Likewise one denotes by $1$ the unique type of order $1$.\newline
If $\sigma$ is a type of order $k$, we define a $\sigma$-flag as a pair $F=(\Gamma,\theta)$, where $\Gamma\in\mathcal{O}$ with $\left|\Gamma \geq k\right|$ and $\theta:\left[k\right]\rightarrow V_\Gamma$ is an injective function such that the labeled vertices define an induced embedding of $\sigma$ into $\Gamma$. An isomorphism between two $\sigma$-flags $(\Gamma,\theta)$ and $(\Gamma',\theta')$ is an isomorphism $\phi$ between $\Gamma$ and $\Gamma'$ where $\phi(\theta(i))=\theta'(i)$. We write $\mathcal{F}^{\sigma}$ for the set of all $\sigma$-flags up to isomorphism. Again, we define $\mathcal{F}^\sigma_\ell\subseteq\mathcal{F}^\sigma$ as the set of all $\sigma$-flags of order $\ell$. For example, $\mathcal{F}^0_\ell=\mathcal{O}_\ell$. If $\sigma$ is a type of order $k$, then $\mathcal{F}^\sigma_k$ consists only of $(\sigma,id)$. One denotes this element simply by $1_\sigma$. \newline
Follow the notation of \cite{flagalgebra}, we write $\mathbf{math}$ $\mathbf{bold}$ $\mathbf{face}$ for denoting random objects.
\begin{definition}(from \cite{flagalgebra})\newline
Fix a type $\sigma$ of order $k$, assume that integers $\ell,\ell_1,\ldots,\ell_t\geq k$ are such that
\begin{equation*}
\ell_1+\ldots+\ell_t-k(t-1)\leq \ell,
\end{equation*}
and $F=(M,\theta)\in \mathcal{F}_\ell^\sigma,\;F_1\in \mathcal{F}_{\ell_1}^\sigma,\ldots,\;\;F_t\in \mathcal{F}_{\ell_t}^\sigma$ are $\sigma$-flags. We define the (key) quantity $p(F_1,\ldots,F_t;F)\in\left[0,1\right]$ as follows. Choose in $V(M)$ uniformely at random a sunflower $(\mathbf{V_1},\ldots,\mathbf{V_t})$ with center $im(\theta)$ and $\forall i\;\left|\mathbf{V_i}\right|=\ell_i$. We let $p(F_1,\ldots,F_t;F)$ denote the probability of the event "$\forall i \in \left[t\right]\;F|_{\mathbf{V_i}}$ is isomorphic to $F_i$." When $t=1$, we use the notation $p(F_1,F)$ instead of $p(F_1;F)$.
\end{definition}
In the following we can identify a $\sigma$-flag $F$ by the probability $p(F,\hat{F})$, where $\hat{F}$ is an arbitrary large enough $\sigma$-flag. Thus, for example if we write
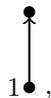
\begin{center}
\begin{tikzpicture}
\filldraw[fill=black] (1,0) circle (0.07cm);
\filldraw[fill=black] (1,1) circle (0.07cm);
\draw[-to,shorten >=2pt,thick] (1,0)--(1,1);
\draw (0.8,0) node {\small{$1$}};
\draw (1.2,-0.1) node {$\;$,};
\end{tikzpicture} 
\end{center}
we can think of it to be the normalized number of outneighbours of a fixed vertex (called "$1$") in an arbitrary large enough orgraph. Or if we write
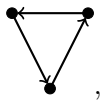
\begin{center}
\begin{tikzpicture}
\filldraw[fill=black] (1,1) circle (0.07cm);
\filldraw[fill=black] (0,1) circle (0.07cm);
\filldraw[fill=black] (0.5,0) circle (0.07cm);
\draw[-to,shorten >=2pt,thick] (0,1)--(0.5,0);
\draw[-to,shorten >=2pt,thick] (0.5,0)--(1,1);
\draw[-to,shorten >=2pt,thick] (1,1)--(0,1);
\end{tikzpicture},
\end{center}
we can think of it to be the density of oriented triangles in an arbitrary large enough orgraph. Remark that these examples are not formal definitions. It should just allow an easier understanding of the following definitions. \newline
Now, we build formal finite linear combinations of $\sigma$-flags. We denote the space which contains these linear combinations by $\mathbb{R}\mathcal{F}^\sigma$. Roughly speaking, if we think of the $F$-density in a graph of sufficently large order for a flag $F\in\mathcal{F}^\sigma_{\ell}$, it seems sensible to call the subspace $\mathcal{K}^\sigma$ which is generated by all elements of the form
\begin{equation*}
F_1-\sum\limits_{\tilde{F}\in\mathcal{F}_{\tilde{\ell}}^\sigma}p(F_1,\tilde{F})\tilde{F},
\end{equation*}
where $F_1\in\mathcal{F}_{\ell_1}^\sigma$ with $\ell_1\leq\tilde{\ell}$, the subspace of "identically zero flag parameters". We want to illustrate this by an example. It can be seen by an easy double-counting argument that
\begin{equation}\label{eq:dcargu}
F_1=\sum\limits_{\tilde{F}\in\mathcal{F}_{\tilde{\ell}}^\sigma}p(F_1,\tilde{F})\tilde{F}.
\end{equation}
For example, the edge-density in an arbitrary large enough orgraph can be expressed as a linear combination of induced subgraph-densities of orgraphs of order $3$ in this graph. Thus,
\begin{equation*}
\parbox[c]{0.3cm}{
\begin{tikzpicture}
\filldraw[fill=black] (1,0) circle (0.07cm);
\filldraw[fill=black] (1,1) circle (0.07cm);

\draw[-to,shorten >=2pt,thick] (1,0)--(1,1);
\end{tikzpicture}}\;=\;
\parbox[c]{1cm}{
\begin{tikzpicture}
\filldraw[fill=black] (1,1) circle (0.07cm);
\filldraw[fill=black] (0,1) circle (0.07cm);
\filldraw[fill=black] (0.5,0) circle (0.07cm);
\draw[-to,shorten >=2pt,thick] (0,1)--(0.5,0);
\draw[-to,shorten >=2pt,thick] (0.5,0)--(1,1);
\draw[-to,shorten >=2pt,thick] (1,1)--(0,1);
\end{tikzpicture}}
\;+\;
\parbox[c]{1cm}{
\begin{tikzpicture}
\filldraw[fill=black] (1,1) circle (0.07cm);
\filldraw[fill=black] (0,1) circle (0.07cm);
\filldraw[fill=black] (0.5,0) circle (0.07cm);
\draw[-to,shorten >=2pt,thick] (0.5,0)--(0,1);
\draw[-to,shorten >=2pt,thick] (0.5,0)--(1,1);
\draw[-to,shorten >=2pt,thick] (1,1)--(0,1);
\end{tikzpicture}}
\;+\;
\frac{2}{3}
\parbox[c]{1cm}{
\begin{tikzpicture}
\filldraw[fill=black] (1,1) circle (0.07cm);
\filldraw[fill=black] (0,1) circle (0.07cm);
\filldraw[fill=black] (0.5,0) circle (0.07cm);
\draw[-to,shorten >=2pt,thick] (0.5,0)--(0,1);
\draw[-to,shorten >=2pt,thick] (0.5,0)--(1,1);
\end{tikzpicture}}
\;+\;
\frac{2}{3}
\parbox[c]{1cm}{
\begin{tikzpicture}
\filldraw[fill=black] (1,1) circle (0.07cm);
\filldraw[fill=black] (0,1) circle (0.07cm);
\filldraw[fill=black] (0.5,0) circle (0.07cm);
\draw[-to,shorten >=2pt,thick] (0.5,0)--(0,1);
\draw[-to,shorten >=2pt,thick] (1,1)--(0.5,0);
\end{tikzpicture}}
\;+\;
\frac{2}{3}
\parbox[c]{1cm}{
\begin{tikzpicture}
\filldraw[fill=black] (1,1) circle (0.07cm);
\filldraw[fill=black] (0,1) circle (0.07cm);
\filldraw[fill=black] (0.5,0) circle (0.07cm);
\draw[-to,shorten >=2pt,thick] (0,1)--(0.5,0);
\draw[-to,shorten >=2pt,thick] (1,1)--(0.5,0);
\end{tikzpicture}}
\;+\;
\frac{1}{3}
\parbox[c]{1cm}{
\begin{tikzpicture}
\filldraw[fill=black] (1,1) circle (0.07cm);
\filldraw[fill=black] (0,1) circle (0.07cm);
\filldraw[fill=black] (0.5,0) circle (0.07cm);
\draw[-to,shorten >=2pt,thick] (1,1)--(0,1);
\end{tikzpicture}}\;\;.
\end{equation*}
Now it is natural to define $\mathcal{A}^\sigma:=\mathbb{R}\mathcal{F}^\sigma / \mathcal{K}^\sigma $ as the flag algebra of the type $\sigma$. This means, we factor $\mathbb{R}\mathcal{F}^\sigma$ by the subspace $\mathcal{K}^\sigma$. In Lemma 2.4 of \cite{flagalgebra} Razborov shows that $\mathcal{A}^\sigma$ is naturally endowed with the structure of a commutative associative algebra. He defines a bilinear mapping for flags in the following way. Let $\sigma$ be a type of order $k$. For two $\sigma$-flags $F_1\in\mathcal{F}_{\ell_1}^\sigma$, $F_2\in\mathcal{F}_{\ell_2}^\sigma$ and $\ell\geq \ell_1+\ell_2-k$ we define
\begin{equation*}
F_1\cdot F_2:=\sum\limits_{F\in\mathcal{F}_\ell^\sigma}p(F_1,F_2;F)F.
\end{equation*}
Remark that this definition is not well defined on $\mathbb{R}\mathcal{F}^\sigma$, but on $\mathcal{A}^\sigma$ it is. The disadvantage of this definition is that this product is just asymptotically the same as the product one would expect, if we interpret the $\sigma$-flags in the above way, because
\begin{equation*}
p(F_1,F_2;F)=p(F_1,F)p(F_2,F)+o(1).
\end{equation*}
That is why flagalgebraic proofs using this product operation are only asymptotically true.\newline
Additionally, we want to remark in a bit crude words, that the function $F\rightarrow p(F,\hat{F})$ for very large $\hat{F}$ asymptotically corresponds to an algebra homomorphism $\phi \in \text{Hom}(\mathcal{A}^\sigma,\mathbb{R})$. Razborov now considers the set
\begin{equation*}
\text{Hom}^+(\mathcal{A}^\sigma,\mathbb{R}):=\left\{\phi \in \text{Hom}(\mathcal{A}^\sigma,\mathbb{R}) | \forall F \in \mathcal{F}^\sigma \; \phi (F)\geq 0\right\}
\end{equation*}
and shows in Corollary 3.4 of \cite{flagalgebra} that $\text{Hom}^+(\mathcal{A}^\sigma,\mathbb{R})$ captures all asymptotically true relations in extremal combinatorics.\newline
Thus, we have seen the basic idea of flag algebras. It is useful to define for $f,g\in \mathcal{A}^0$ that $f\geq g$, if $\forall \phi \in \text{Hom}^+(\mathcal{A}^\sigma,\mathbb{R}) \left(\phi(f)\geq\phi(g)\right)$. This is a partial preorder on $\mathcal{A}^0$. Now we want to turn our attention to an application of the Cauchy-Schwarz inequality in flag algebras.  \newline
We define the averaging operator $\left\llbracket \cdot \right\rrbracket_\sigma\;:\;\mathcal{A}^\sigma\rightarrow \mathcal{A}^0$ as follows. For a type $\sigma$ of order $k$ and $F=(\Gamma,\theta)\in\mathcal{F}^\sigma$, let $q_\sigma(F)$ be the probability that a uniformely at random chosen injective mapping $\mbox{\boldmath$\theta$}:\;\left[k\right]\rightarrow V_\Gamma$ defines an induced embedding of $\sigma$ in $\Gamma$ and the resulting $\sigma$-flag $(\Gamma,\mbox{\boldmath$\theta$})$ is isomorphic to $F$. Now, we define 
\begin{equation*}
\left\llbracket F\right\rrbracket_\sigma:=q_{\sigma}(F)\cdot \Gamma
\end{equation*}
partially on $\mathcal{F}^\sigma$. 
In section 2.2 in \cite{flagalgebra}, Razborov proves that this operator can be extended linearly to $\mathcal{A}^\sigma$ and he explains why it corresponds to averaging.
\begin{theorem}\label{th:cs}\textit{Cauchy-Schwarz inequality}(from \cite{flagalgebra}, Theorem 3.14)\newline
Let $f,g\in\mathcal{F}^\sigma$, then
\begin{equation*}
\left\llbracket f^2\right\rrbracket_\sigma \cdot\left\llbracket g^2\right\rrbracket_\sigma \geq \left\llbracket fg\right\rrbracket_\sigma^2.
\end{equation*}
In particular $(g=1_\sigma)$,
\begin{equation*}
\left\llbracket f^2\right\rrbracket_\sigma \cdot \sigma \geq \left\llbracket f\right\rrbracket_\sigma^2,
\end{equation*}
which in turn implies
\begin{equation*}
\left\llbracket f^2\right\rrbracket_\sigma\geq 0.
\end{equation*}
\end{theorem}
It is an easy consequence of Theorem~\ref{th:cs} that $\left\llbracket\underline{v}^TA\underline{v}\right\rrbracket_\sigma\geq 0$, if $\underline{v}$ is a vector of $n$ $\sigma$-flags and $A\in\mathbb{R}^{n\times n}$ is symmetric positive semidefinite.
\begin{example}
As an example, we want to show with the described method of flag algebras that $i(\vec P_3)\leq \frac{4}{7}$ and that $i(\vec C_3)=\frac{1}{4}$. For this purpose we look at the following equalities.
\begin{equation*}
1  = 
\parbox[c]{1cm}{\begin{tikzpicture}
\filldraw[fill=black] (1,1) circle (0.07cm);
\filldraw[fill=black] (0,1) circle (0.07cm);
\filldraw[fill=black] (0.5,0) circle (0.07cm);
\end{tikzpicture}}
\;\;+\;
\parbox[c]{1cm}{\begin{tikzpicture}
\filldraw[fill=black] (1,1) circle (0.07cm);
\filldraw[fill=black] (0,1) circle (0.07cm);
\filldraw[fill=black] (0.5,0) circle (0.07cm);
\draw[-to,shorten >=2pt,thick] (0.5,0)--(1,1);
\end{tikzpicture}}
\;\;+\;
\parbox[c]{1cm}{\begin{tikzpicture}
\filldraw[fill=black] (1,1) circle (0.07cm);
\filldraw[fill=black] (0,1) circle (0.07cm);
\filldraw[fill=black] (0.5,0) circle (0.07cm);
\draw[-to,shorten >=2pt,thick] (0.5,0)--(1,1);
\draw[-to,shorten >=2pt,thick] (0.5,0)--(0,1);
\end{tikzpicture}}
\;\;+\;
\parbox[c]{1cm}{\begin{tikzpicture}
\filldraw[fill=black] (1,1) circle (0.07cm);
\filldraw[fill=black] (0,1) circle (0.07cm);
\filldraw[fill=black] (0.5,0) circle (0.07cm);
\draw[-to,shorten >=2pt,thick] (0.5,0)--(1,1);
\draw[-to,shorten >=2pt,thick] (0,1)--(0.5,0);
\end{tikzpicture}}
\;\;+\;
\parbox[c]{1cm}{\begin{tikzpicture}
\filldraw[fill=black] (1,1) circle (0.07cm);
\filldraw[fill=black] (0,1) circle (0.07cm);
\filldraw[fill=black] (0.5,0) circle (0.07cm);
\draw[-to,shorten >=2pt,thick] (1,1)--(0.5,0);
\draw[-to,shorten >=2pt,thick] (0,1)--(0.5,0);
\end{tikzpicture}}
\;\;+\;
\parbox[c]{1cm}{\begin{tikzpicture}
\filldraw[fill=black] (1,1) circle (0.07cm);
\filldraw[fill=black] (0,1) circle (0.07cm);
\filldraw[fill=black] (0.5,0) circle (0.07cm);
\draw[-to,shorten >=2pt,thick] (0.5,0)--(1,1);
\draw[-to,shorten >=2pt,thick] (0.5,0)--(0,1);
\draw[-to,shorten >=2pt,thick] (1,1)--(0,1);
\end{tikzpicture}}
\;\;+\;
\parbox[c]{1cm}{\begin{tikzpicture}
\filldraw[fill=black] (1,1) circle (0.07cm);
\filldraw[fill=black] (0,1) circle (0.07cm);
\filldraw[fill=black] (0.5,0) circle (0.07cm);
\draw[-to,shorten >=2pt,thick] (0.5,0)--(1,1);
\draw[-to,shorten >=2pt,thick] (0,1)--(0.5,0);
\draw[-to,shorten >=2pt,thick] (1,1)--(0,1);
\end{tikzpicture}} 
\end{equation*}
\begin{eqnarray*}
\left\llbracket\,\left(
\parbox[c]{0.7cm}{
\begin{tikzpicture}
\filldraw[fill=black] (1,0) circle (0.07cm);
\filldraw[fill=black] (1,1) circle (0.07cm);
\draw (0.8,0) node {\small{$1$}};
\end{tikzpicture}}
\right)^2\,\right\rrbracket_1
\;=\;\quad
\left\llbracket \,
\parbox[c]{1cm}{
\begin{tikzpicture}
\filldraw[fill=black] (1,1) circle (0.07cm);
\filldraw[fill=black] (0,1) circle (0.07cm);
\filldraw[fill=black] (0.5,0) circle (0.07cm);
\draw (0.3,0) node {\small{$1$}};
\end{tikzpicture}}+
\parbox[c]{1cm}{\begin{tikzpicture}
\filldraw[fill=black] (1,1) circle (0.07cm);
\filldraw[fill=black] (0,1) circle (0.07cm);
\filldraw[fill=black] (0.5,0) circle (0.07cm);
\draw[-to,shorten >=2pt,thick] (1,1)--(0,1);
\draw (0.3,0) node {\small{$1$}};
\end{tikzpicture}}
\; \; \right\rrbracket_1 & = & \;\;
\parbox[c]{1cm}{\begin{tikzpicture}
\filldraw[fill=black] (1,1) circle (0.07cm);
\filldraw[fill=black] (0,1) circle (0.07cm);
\filldraw[fill=black] (0.5,0) circle (0.07cm);
\end{tikzpicture}}
\;\;+\;\frac{1}{3}
\parbox[c]{1cm}{\begin{tikzpicture}
\filldraw[fill=black] (1,1) circle (0.07cm);
\filldraw[fill=black] (0,1) circle (0.07cm);
\filldraw[fill=black] (0.5,0) circle (0.07cm);
\draw[-to,shorten >=2pt,thick] (0.5,0)--(1,1);
\end{tikzpicture}} \\
\frac{3}{2}
\left\llbracket\,\left(
\parbox[c]{0.7cm}{
\begin{tikzpicture}
\filldraw[fill=black] (1,0) circle (0.07cm);
\filldraw[fill=black] (1,1) circle (0.07cm);
\draw[-to,shorten >=2pt,thick] (1,0)--(1,1);
\draw (0.8,0) node {\small{$1$}};
\end{tikzpicture}}
\right)^2\,\right\rrbracket_1
\;=\;
\frac{3}{2}\left\llbracket \,
\parbox[c]{1cm}{
\begin{tikzpicture}
\filldraw[fill=black] (1,1) circle (0.07cm);
\filldraw[fill=black] (0,1) circle (0.07cm);
\filldraw[fill=black] (0.5,0) circle (0.07cm);
\draw[-to,shorten >=2pt,thick] (0.5,0)--(1,1);
\draw[-to,shorten >=2pt,thick] (0.5,0)--(0,1);

\draw (0.3,0) node {\small{$1$}};
\end{tikzpicture}}
+
\parbox[c]{1cm}{\begin{tikzpicture}
\filldraw[fill=black] (1,1) circle (0.07cm);
\filldraw[fill=black] (0,1) circle (0.07cm);
\filldraw[fill=black] (0.5,0) circle (0.07cm);
\draw[-to,shorten >=2pt,thick] (0.5,0)--(1,1);
\draw[-to,shorten >=2pt,thick] (0.5,0)--(0,1);
\draw[-to,shorten >=2pt,thick] (1,1)--(0,1);
\draw (0.3,0) node {\small{$1$}};
\end{tikzpicture}}
\; \; \right\rrbracket_1 & = & \frac{1}{2}
\parbox[c]{1cm}{\begin{tikzpicture}
\filldraw[fill=black] (1,1) circle (0.07cm);
\filldraw[fill=black] (0,1) circle (0.07cm);
\filldraw[fill=black] (0.5,0) circle (0.07cm);
\draw[-to,shorten >=2pt,thick] (0.5,0)--(1,1);
\draw[-to,shorten >=2pt,thick] (0.5,0)--(0,1);
\end{tikzpicture}}
\;\;+\;\frac{1}{2}
\parbox[c]{1cm}{\begin{tikzpicture}
\filldraw[fill=black] (1,1) circle (0.07cm);
\filldraw[fill=black] (0,1) circle (0.07cm);
\filldraw[fill=black] (0.5,0) circle (0.07cm);
\draw[-to,shorten >=2pt,thick] (0.5,0)--(1,1);
\draw[-to,shorten >=2pt,thick] (0.5,0)--(0,1);
\draw[-to,shorten >=2pt,thick] (1,1)--(0,1);
\end{tikzpicture}} \\
\frac{3}{2}
\left\llbracket\,\left(
\parbox[c]{0.7cm}{
\begin{tikzpicture}
\filldraw[fill=black] (1,0) circle (0.07cm);
\filldraw[fill=black] (1,1) circle (0.07cm);
\draw[-to,shorten >=2pt,thick] (1,1)--(1,0);
\draw (0.8,0) node {\small{$1$}};
\end{tikzpicture}}
\right)^2\,\right\rrbracket_1
\;=\;
\frac{3}{2}
\left\llbracket \,
\parbox[c]{1cm}{
\begin{tikzpicture}
\filldraw[fill=black] (1,1) circle (0.07cm);
\filldraw[fill=black] (0,1) circle (0.07cm);
\filldraw[fill=black] (0.5,0) circle (0.07cm);
\draw[-to,shorten >=2pt,thick] (1,1)--(0.5,0);
\draw[-to,shorten >=2pt,thick] (0,1)--(0.5,0);

\draw (0.3,0) node {\small{$1$}};
\end{tikzpicture}}
+
\parbox[c]{1cm}{\begin{tikzpicture}
\filldraw[fill=black] (1,1) circle (0.07cm);
\filldraw[fill=black] (0,1) circle (0.07cm);
\filldraw[fill=black] (0.5,0) circle (0.07cm);
\draw[-to,shorten >=2pt,thick] (1,1)--(0.5,0);
\draw[-to,shorten >=2pt,thick] (0,1)--(0.5,0);
\draw[-to,shorten >=2pt,thick] (1,1)--(0,1);
\draw (0.3,0) node {\small{$1$}};
\end{tikzpicture}}
\; \; \right\rrbracket_1 & = & \frac{1}{2}
\parbox[c]{1cm}{\begin{tikzpicture}
\filldraw[fill=black] (1,1) circle (0.07cm);
\filldraw[fill=black] (0,1) circle (0.07cm);
\filldraw[fill=black] (0.5,0) circle (0.07cm);
\draw[-to,shorten >=2pt,thick] (1,1)--(0.5,0);
\draw[-to,shorten >=2pt,thick] (0,1)--(0.5,0);
\end{tikzpicture}}
\;\;+\;\frac{1}{2}
\parbox[c]{1cm}{\begin{tikzpicture}
\filldraw[fill=black] (1,1) circle (0.07cm);
\filldraw[fill=black] (0,1) circle (0.07cm);
\filldraw[fill=black] (0.5,0) circle (0.07cm);
\draw[-to,shorten >=2pt,thick] (0.5,0)--(1,1);
\draw[-to,shorten >=2pt,thick] (0.5,0)--(0,1);
\draw[-to,shorten >=2pt,thick] (1,1)--(0,1);
\end{tikzpicture}}
\end{eqnarray*}
Now these equations tell us everything we need. For an easier notation we denote the flag
\begin{tikzpicture}[scale=0.3]
\filldraw[fill=black] (1,0) circle (0.07cm);
\filldraw[fill=black] (1,1) circle (0.07cm);
\draw[-to,shorten >=1pt] (1,0)--(1,1);
\end{tikzpicture} by $\rho$.
Notice that each evaluation of $\rho$ with an orgraph homomorphism $\phi\in\text{Hom}^+(\mathcal{A}^\sigma,\mathbb{R})$ belong to the edge-density of the orgraph "corresponding" to $\phi$.\newline
Now, we have
\begin{eqnarray*}
\parbox[c]{1cm}{\begin{tikzpicture}
\filldraw[fill=black] (1,1) circle (0.07cm);
\filldraw[fill=black] (0,1) circle (0.07cm);
\filldraw[fill=black] (0.5,0) circle (0.07cm);
\draw[-to,shorten >=2pt,thick] (0.5,0)--(1,1);
\draw[-to,shorten >=2pt,thick] (0,1)--(0.5,0);
\end{tikzpicture}}
\;\;+\;
\parbox[c]{1cm}{\begin{tikzpicture}
\filldraw[fill=black] (1,1) circle (0.07cm);
\filldraw[fill=black] (0,1) circle (0.07cm);
\filldraw[fill=black] (0.5,0) circle (0.07cm);
\draw[-to,shorten >=2pt,thick] (0.5,0)--(1,1);
\draw[-to,shorten >=2pt,thick] (0,1)--(0.5,0);
\draw[-to,shorten >=2pt,thick] (1,1)--(0,1);
\end{tikzpicture}}
& \leq & 1 \;-\;
\left\llbracket\,\left(
\parbox[c]{0.7cm}{
\begin{tikzpicture}
\filldraw[fill=black] (1,0) circle (0.07cm);
\filldraw[fill=black] (1,1) circle (0.07cm);
\draw (0.8,0) node {\small{$1$}};
\end{tikzpicture}}
\right)^2\,\right\rrbracket_1
\;-\;
\frac{3}{2}\left\llbracket\,\left(
\parbox[c]{0.7cm}{
\begin{tikzpicture}
\filldraw[fill=black] (1,0) circle (0.07cm);
\filldraw[fill=black] (1,1) circle (0.07cm);
\draw[-to,shorten >=2pt,thick] (1,0)--(1,1);
\draw (0.8,0) node {\small{$1$}};
\end{tikzpicture}}
\right)^2\,\right\rrbracket_1
\;-\;
\frac{3}{2}\left\llbracket\,\left(
\parbox[c]{0.7cm}{
\begin{tikzpicture}
\filldraw[fill=black] (1,0) circle (0.07cm);
\filldraw[fill=black] (1,1) circle (0.07cm);
\draw[-to,shorten >=2pt,thick] (1,1)--(1,0);
\draw (0.8,0) node {\small{$1$}};
\end{tikzpicture}}
\right)^2\,\right\rrbracket_1 \\
& \leq & 1-(1-\rho)^2-\frac{3}{2}\left(\frac{\rho}{2}\right)^2-\frac{3}{2}\left(\frac{\rho}{2}\right)^2=-\frac{7}{4}\rho^2+2\rho
\end{eqnarray*}
The righthandside depends just on the edge-density $\rho$, which can be minimized by taking $\rho=\frac{4}{7}$. Thus, we have
\begin{equation*}
\parbox[c]{1cm}{\begin{tikzpicture}
\filldraw[fill=black] (1,1) circle (0.07cm);
\filldraw[fill=black] (0,1) circle (0.07cm);
\filldraw[fill=black] (0.5,0) circle (0.07cm);
\draw[-to,shorten >=2pt,thick] (0.5,0)--(1,1);
\draw[-to,shorten >=2pt,thick] (0,1)--(0.5,0);
\end{tikzpicture}}
\;\;+\;
\parbox[c]{1cm}{\begin{tikzpicture}
\filldraw[fill=black] (1,1) circle (0.07cm);
\filldraw[fill=black] (0,1) circle (0.07cm);
\filldraw[fill=black] (0.5,0) circle (0.07cm);
\draw[-to,shorten >=2pt,thick] (0.5,0)--(1,1);
\draw[-to,shorten >=2pt,thick] (0,1)--(0.5,0);
\draw[-to,shorten >=2pt,thick] (1,1)--(0,1);
\end{tikzpicture}}
\; \leq \; \frac{4}{7} \Rightarrow
\parbox[c]{1cm}{\begin{tikzpicture}
\filldraw[fill=black] (1,1) circle (0.07cm);
\filldraw[fill=black] (0,1) circle (0.07cm);
\filldraw[fill=black] (0.5,0) circle (0.07cm);
\draw[-to,shorten >=2pt,thick] (0.5,0)--(1,1);
\draw[-to,shorten >=2pt,thick] (0,1)--(0.5,0);
\end{tikzpicture}}
\;\leq\;\frac{4}{7} \Rightarrow i(\vec P_3)\leq\frac{4}{7}.
\end{equation*}
Additionally, we know that
\begin{equation*}
\parbox[c]{1cm}{\begin{tikzpicture}
\filldraw[fill=black] (1,1) circle (0.07cm);
\filldraw[fill=black] (0,1) circle (0.07cm);
\filldraw[fill=black] (0.5,0) circle (0.07cm);
\draw[-to,shorten >=2pt,thick] (0.5,0)--(1,1);
\draw[-to,shorten >=2pt,thick] (0,1)--(0.5,0);
\draw[-to,shorten >=2pt,thick] (1,1)--(0,1);
\end{tikzpicture}}
\;\leq\;-\frac{7}{4}\rho^2+2\rho.
\end{equation*}
We can assume, that there are extremal orgraphs for $i(\vec C_3)$, where the edge-density $\rho=1$, because in every extremal orgraph we can fill the missing edges in an abitrary way by new edges without decreasing the number of $\vec C_3$ in this extremal orgraph. Thus, we get
\begin{equation*}
\parbox[c]{1cm}{\begin{tikzpicture}
\filldraw[fill=black] (1,1) circle (0.07cm);
\filldraw[fill=black] (0,1) circle (0.07cm);
\filldraw[fill=black] (0.5,0) circle (0.07cm);
\draw[-to,shorten >=2pt,thick] (0.5,0)--(1,1);
\draw[-to,shorten >=2pt,thick] (0,1)--(0.5,0);
\draw[-to,shorten >=2pt,thick] (1,1)--(0,1);
\end{tikzpicture}}
\;\leq\;\frac{1}{4} \Rightarrow i(\vec C_3)\leq \frac{1}{4}.
\end{equation*}
On the other hand, if we have a look at $\lim\limits_{n\rightarrow\infty}\left(\vec C_3\right)^{\circ n}$, it is easy to see that $i(\vec C_3)\geq\frac{1}{4}$. 
\end{example}
\begin{example}
In our second example we show that $i(\vec K_2 \cup \vec E_1)=\frac{3}{4}$. Let $T_n$ be an arbitrary complete orgraph on $n$ vertices, thus a tournament. By a look at the limit graph $\lim\limits_{n \to \infty}T_n\cup T_n$ its easy to see that $i(\vec K_2 \cup \vec E_1)\geq\frac{3}{4}$. We define a vector of $1$-flags as
\begin{equation*}
\underline{g}:=\left(
\parbox[c]{1cm}{
\begin{tikzpicture}
\filldraw[fill=black] (1,0) circle (0.07cm);
\filldraw[fill=black] (1,1) circle (0.07cm);
\draw (0.8,0) node {\small{$1$}};
\end{tikzpicture}},
\parbox[c]{1cm}{
\begin{tikzpicture}
\filldraw[fill=black] (1,0) circle (0.07cm);
\filldraw[fill=black] (1,1) circle (0.07cm);
\draw[-to,shorten >=2pt,thick] (1,0)--(1,1);
\draw (0.8,0) node {\small{$1$}};
\end{tikzpicture}},
\parbox[c]{1cm}{
\begin{tikzpicture}
\filldraw[fill=black] (1,0) circle (0.07cm);
\filldraw[fill=black] (1,1) circle (0.07cm);
\draw[-to,shorten >=2pt,thick] (1,1)--(1,0);
\draw (0.8,0) node {\small{$1$}};
\end{tikzpicture}}
\right)^T.
\end{equation*}
Now, we get with the help of a positive semidefinite matrix
\begin{eqnarray*}
\frac{3}{4}-\;
\parbox[c]{1cm}{\begin{tikzpicture}
\filldraw[fill=black] (1,1) circle (0.07cm);
\filldraw[fill=black] (0,1) circle (0.07cm);
\filldraw[fill=black] (0.5,0) circle (0.07cm);
\draw[-to,shorten >=2pt,thick] (0.5,0)--(1,1);
\end{tikzpicture}} 
& \geq &
\frac{3}{4} 
\left\llbracket \underline{g}^T 
\left(
\begin{array}{*{3}{c}}
 1 & -1 & -1 \\
-1 &  1 &  1 \\
-1 &  1 &  1
\end{array}
\right)
\underline{g}
\right\rrbracket_1 \;\geq\;0 \\
\Rightarrow i(\vec K_2 \cup \vec E_1) & \leq & \frac{3}{4}.
\end{eqnarray*}
\end{example}
\section{Main Results}
\label{sec:mainresult}
A lot of calculations in our proofs deal with the $582$ elements of $\mathcal{O}_5=\mathcal{F}_5^0$ and the $15$ elements of $\mathcal{F}_3^1$ which are defined in figure \ref{fig:Fdef}. Let $\underline{f^1}$ be the vector with $\underline{f^1}(i):=F^1_i$ for $i\in\left\{0,1,\ldots,14\right\}$. 
\begin{figure}
\begin{center}
\begin{tikzpicture}
\filldraw[fill=black] (1,1) circle (0.07cm);
\filldraw[fill=black] (0,1) circle (0.07cm);
\filldraw[fill=black] (0.5,0) circle (0.07cm);
\draw (0.3,0) node {\small{$1$}};
\draw  (0.5,-0.5) node {$\vec F^1_0$};
\end{tikzpicture}
$\quad$
\begin{tikzpicture}
\filldraw[fill=black] (1,1) circle (0.07cm);
\filldraw[fill=black] (0,1) circle (0.07cm);
\filldraw[fill=black] (0.5,0) circle (0.07cm);
\draw (0.3,0) node {\small{$1$}};
\draw  (0.5,-0.5) node {$\vec F^1_1$};
\draw[-to,shorten >=2pt,thick] (0.5,0)--(1,1);
\end{tikzpicture}$\quad$
\begin{tikzpicture}
\filldraw[fill=black] (1,1) circle (0.07cm);
\filldraw[fill=black] (0,1) circle (0.07cm);
\filldraw[fill=black] (0.5,0) circle (0.07cm);
\draw (0.3,0) node {\small{$1$}};
\draw  (0.5,-0.5) node {$\vec F^1_2$};
\draw[-to,shorten >=2pt,thick] (0.5,0)--(1,1);
\draw[-to,shorten >=2pt,thick] (0.5,0)--(0,1);
\end{tikzpicture}$\quad$
\begin{tikzpicture}
\filldraw[fill=black] (1,1) circle (0.07cm);
\filldraw[fill=black] (0,1) circle (0.07cm);
\filldraw[fill=black] (0.5,0) circle (0.07cm);
\draw (0.3,0) node {\small{$1$}};
\draw  (0.5,-0.5) node {$\vec F^1_3$};
\draw[-to,shorten >=2pt,thick] (0.5,0)--(1,1);
\draw[-to,shorten >=2pt,thick] (0.5,0)--(0,1);
\draw[-to,shorten >=2pt,thick] (1,1)--(0,1);
\end{tikzpicture}$\quad$
\begin{tikzpicture}
\filldraw[fill=black] (1,1) circle (0.07cm);
\filldraw[fill=black] (0,1) circle (0.07cm);
\filldraw[fill=black] (0.5,0) circle (0.07cm);
\draw (0.3,0) node {\small{$1$}};
\draw  (0.5,-0.5) node {$\vec F^1_4$};
\draw[-to,shorten >=2pt,thick] (0.5,0)--(1,1);
\draw[-to,shorten >=2pt,thick] (0,1)--(0.5,0);
\end{tikzpicture}$\quad$
\begin{tikzpicture}
\filldraw[fill=black] (1,1) circle (0.07cm);
\filldraw[fill=black] (0,1) circle (0.07cm);
\filldraw[fill=black] (0.5,0) circle (0.07cm);
\draw (0.3,0) node {\small{$1$}};
\draw  (0.5,-0.5) node {$\vec F^1_5$};
\draw[-to,shorten >=2pt,thick] (0.5,0)--(1,1);
\draw[-to,shorten >=2pt,thick] (0,1)--(0.5,0);
\draw[-to,shorten >=2pt,thick] (1,1)--(0,1);
\end{tikzpicture}$\quad$
\begin{tikzpicture}
\filldraw[fill=black] (1,1) circle (0.07cm);
\filldraw[fill=black] (0,1) circle (0.07cm);
\filldraw[fill=black] (0.5,0) circle (0.07cm);
\draw (0.3,0) node {\small{$1$}};
\draw  (0.5,-0.5) node {$\vec F^1_6$};
\draw[-to,shorten >=2pt,thick] (0.5,0)--(1,1);
\draw[-to,shorten >=2pt,thick] (0,1)--(0.5,0);
\draw[-to,shorten >=2pt,thick] (0,1)--(1,1);
\end{tikzpicture}$\quad$
\begin{tikzpicture}
\filldraw[fill=black] (1,1) circle (0.07cm);
\filldraw[fill=black] (0,1) circle (0.07cm);
\filldraw[fill=black] (0.5,0) circle (0.07cm);
\draw (0.3,0) node {\small{$1$}};
\draw  (0.5,-0.5) node {$\vec F^1_7$};
\draw[-to,shorten >=2pt,thick] (0.5,0)--(1,1);
\draw[-to,shorten >=2pt,thick] (1,1)--(0,1);
\end{tikzpicture}$\quad$
\begin{tikzpicture}
\filldraw[fill=black] (1,1) circle (0.07cm);
\filldraw[fill=black] (0,1) circle (0.07cm);
\filldraw[fill=black] (0.5,0) circle (0.07cm);
\draw (0.3,0) node {\small{$1$}};
\draw  (0.5,-0.5) node {$\vec F^1_8$};
\draw[-to,shorten >=2pt,thick] (0.5,0)--(1,1);
\draw[-to,shorten >=2pt,thick] (0,1)--(1,1);
\end{tikzpicture}$\quad$
\begin{tikzpicture}
\filldraw[fill=black] (1,1) circle (0.07cm);
\filldraw[fill=black] (0,1) circle (0.07cm);
\filldraw[fill=black] (0.5,0) circle (0.07cm);
\draw (0.3,0) node {\small{$1$}};
\draw  (0.5,-0.5) node {$\vec F^1_9$};
\draw[-to,shorten >=2pt,thick] (1,1)--(0.5,0);
\end{tikzpicture}$\quad$
\begin{tikzpicture}
\filldraw[fill=black] (1,1) circle (0.07cm);
\filldraw[fill=black] (0,1) circle (0.07cm);
\filldraw[fill=black] (0.5,0) circle (0.07cm);
\draw (0.3,0) node {\small{$1$}};
\draw  (0.5,-0.5) node {$\vec F^1_{10}$};
\draw[-to,shorten >=2pt,thick] (1,1)--(0.5,0);
\draw[-to,shorten >=2pt,thick] (0,1)--(0.5,0);
\end{tikzpicture}$\quad$
\begin{tikzpicture}
\filldraw[fill=black] (1,1) circle (0.07cm);
\filldraw[fill=black] (0,1) circle (0.07cm);
\filldraw[fill=black] (0.5,0) circle (0.07cm);
\draw (0.3,0) node {\small{$1$}};
\draw  (0.5,-0.5) node {$\vec F^1_{11}$};
\draw[-to,shorten >=2pt,thick] (1,1)--(0.5,0);
\draw[-to,shorten >=2pt,thick] (0,1)--(0.5,0);
\draw[-to,shorten >=2pt,thick] (1,1)--(0,1);
\end{tikzpicture}$\quad$
\begin{tikzpicture}
\filldraw[fill=black] (1,1) circle (0.07cm);
\filldraw[fill=black] (0,1) circle (0.07cm);
\filldraw[fill=black] (0.5,0) circle (0.07cm);
\draw (0.3,0) node {\small{$1$}};
\draw  (0.5,-0.5) node {$\vec F^1_{12}$};
\draw[-to,shorten >=2pt,thick] (1,1)--(0.5,0);
\draw[-to,shorten >=2pt,thick] (1,1)--(0,1);
\end{tikzpicture}$\quad$
\begin{tikzpicture}
\filldraw[fill=black] (1,1) circle (0.07cm);
\filldraw[fill=black] (0,1) circle (0.07cm);
\filldraw[fill=black] (0.5,0) circle (0.07cm);
\draw (0.3,0) node {\small{$1$}};
\draw  (0.5,-0.5) node {$\vec F^1_{13}$};
\draw[-to,shorten >=2pt,thick] (1,1)--(0.5,0);
\draw[-to,shorten >=2pt,thick] (0,1)--(1,1);
\end{tikzpicture}$\quad$
\begin{tikzpicture}
\filldraw[fill=black] (1,1) circle (0.07cm);
\filldraw[fill=black] (0,1) circle (0.07cm);
\filldraw[fill=black] (0.5,0) circle (0.07cm);
\draw (0.3,0) node {\small{$1$}};
\draw  (0.5,-0.5) node {$\vec F^1_{14}$};
\draw[-to,shorten >=2pt,thick] (1,1)--(0,1);
\end{tikzpicture}
\end{center}
\caption{The $15$ elements of $\mathcal{F}_3^1$.}\label{fig:Fdef}
\end{figure}  
\begin{theorem}\label{th:p3c4}
\begin{eqnarray*}
i(\vec P_3) & \leq & 0.4446 \\
i(\vec C_4) & \leq & 0.1104
\end{eqnarray*}
\end{theorem}
\begin{proof}
We define the symmetric matrices $A,B\in \mathbb{R}^{15\times 15}$ by
\footnotesize
\begin{equation*}
A:=\frac{6}{10^{4}}
\left(
\begin{array}{*{8}{r}}
       739 &     1153 &  -    62 &  -   420 &  -   120 &  -    31 &  -   321 &   \\
      1153 &     7013 &     2254 &     2355 &  -    18 &     4192 &  -   772 &   \\
   -    62 &     2254 &     3147 &     2726 &     1998 &      925 &  -   142 &   \\
   -   420 &     2355 &     2726 &     6798 &  -   488 &      545 &      804 &   \\
   -   120 &  -    18 &     1998 &  -   488 &     4573 &  -  2496 &  -   529 &   \\
   -    31 &     4192 &      925 &      545 &  -  2496 &    12610 &      777 &   \\
   -   321 &  -   772 &  -   142 &      804 &  -   529 &      777 &     1578 & \ldots  \\
   -   131 &       51 &      420 &     2622 &  -  1719 &  -  1327 &  -   495 &   \\
   -    91 &     1992 &      865 &     3770 &  -  2159 &     3317 &      564 &   \\
      1153 &     1953 &  -   636 &  -  3913 &  -    18 &     4192 &  -   772 &   \\
   -    62 &  -   636 &  -   469 &  -  2970 &     1998 &      925 &  -   142 &   \\
   -   420 &  -  3913 &  -  2970 &  -  5390 &  -   488 &      545 &      804 &   \\
   -    91 &  -   798 &  -  1853 &  -  3172 &  -  2159 &     3317 &      564 &   \\
   -   131 &  -  2009 &  -  2958 &  -  2781 &  -  1719 &  -  1327 &  -   495 &   \\
       663 &     1276 &  -  2042 &  -   527 &  -  4884 &     3548 &  -   491 &  
 \end{array}
\right. 
\end{equation*}
\begin{equation*}
\left.
\begin{array}{*{8}{r}}
 -   131 &  -    91 &     1153 &  -    62 &  -   420 &  -    91 &  -   131 &      663 \\
      51 &     1992 &     1953 &  -   636 &  -  3913 &  -   798 &  -  2009 &     1276 \\
     420 &      865 &  -   636 &  -   469 &  -  2970 &  -  1853 &  -  2958 &  -  2042 \\
    2622 &     3770 &  -  3913 &  -  2970 &  -  5390 &  -  3172 &  -  2781 &  -   527 \\
 -  1719 &  -  2159 &  -    18 &     1998 &  -   488 &  -  2159 &  -  1719 &  -  4884 \\
 -  1327 &     3317 &     4192 &      925 &      545 &     3317 &  -  1327 &     3548 \\
 -   495 &      564 &  -   772 &  -   142 &      804 &      564 &  -   495 &  -   491 \\
    4221 &     1590 &  -  2009 &  -  2958 &  -  2781 &  -  1067 &  -   992 &     2025 \\
    1590 &     4666 &  -   798 &  -  1853 &  -  3172 &  -  1329 &  -  1067 &     2603 \\
 -  2009 &  -   798 &     7013 &     2254 &     2355 &     1992 &       51 &     1276 \\
 -  2958 &  -  1853 &     2254 &     3147 &     2726 &      865 &      420 &  -  2042 \\
 -  2781 &  -  3172 &     2355 &     2726 &     6798 &     3770 &     2622 &  -   527 \\
 -  1067 &  -  1329 &     1992 &      865 &     3770 &     4666 &     1590 &     2603 \\
 -   992 &  -  1067 &       51 &      420 &     2622 &     1590 &     4221 &     2025 \\
    2025 &     2603 &     1276 &  -  2042 &  -   527 &     2603 &     2025 &     8134 
\end{array}\right)  
\end{equation*}
\begin{equation*}
B:=\frac{6}{10^{4}}
\left(
\begin{array}{*{8}{r}}
       183 &      408 &  -   148 &  -    26 &  -   588 &  -     7 &      386 &    \\
       408 &     1549 &  -    43 &      485 &  -  1637 &      397 &      932 &    \\      
   -   148 &  -    43 &      620 &      455 &     1078 &       22 &  -   446 &    \\      
   -    26 &      485 &      455 &     1093 &      334 &      201 &      303 &    \\      
   -   588 &  -  1637 &     1078 &      334 &     5261 &  -   298 &  -  1060 &    \\      
   -     7 &      397 &       22 &      201 &  -   298 &     1084 &      178 &    \\      
       386 &      932 &  -   446 &      303 &  -  1060 &      178 &     1654 &  \ldots  \\
   -     1 &      684 &      561 &     1369 &  -  1931 &  -   320 &  -   744 &    \\      
        83 &      754 &       30 &      942 &  -   265 &      610 &      698 &    \\      
       408 &      999 &  -   569 &  -   356 &  -  1637 &      397 &      932 &    \\      
   -   148 &  -   569 &  -    33 &  -   390 &     1078 &       22 &  -   446 &    \\      
   -    26 &  -   356 &  -   390 &  -   672 &      334 &      201 &      303 &    \\      
        83 &      541 &  -   118 &  -   238 &  -   265 &      610 &      698 &    \\      
   -     1 &  -  1036 &  -  1180 &  -  2401 &  -  1931 &  -   320 &  -   744 &    \\      
       178 &      302 &  -   492 &  -   412 &  -  2498 &  -   205 &  -    93 & 
\end{array}
\right. 
\end{equation*}
\begin{equation*}
\left.
\begin{array}{*{8}{r}}
 -     1 &       83 &      408 &  -   148 &  -    26 &       83 &  -     1 &      178\\
     684 &      754 &      999 &  -   569 &  -   356 &      541 &  -  1036 &      302\\
     561 &       30 &  -   569 &  -    33 &  -   390 &  -   118 &  -  1180 &  -   492\\
    1369 &      942 &  -   356 &  -   390 &  -   672 &  -   238 &  -  2401 &  -   412\\
 -  1931 &  -   265 &  -  1637 &     1078 &      334 &  -   265 &  -  1931 &  -  2498\\
 -   320 &      610 &      397 &       22 &      201 &      610 &  -   320 &  -   205\\
 -   744 &      698 &      932 &  -   446 &      303 &      698 &  -   744 &  -    93\\
    5941 &      589 &  -  1036 &  -  1180 &  -  2401 &  -  2154 &  -  2190 &     1614\\
     589 &     1847 &      541 &  -   118 &  -   238 &  -    56 &  -  2154 &  -   468\\
 -  1036 &      541 &     1549 &  -    43 &      485 &      754 &      684 &      302\\
 -  1180 &  -   118 &  -    43 &      620 &      455 &       30 &      561 &  -   492\\
 -  2401 &  -   238 &      485 &      455 &     1093 &      942 &     1369 &  -   412\\
 -  2154 &  -    56 &      754 &       30 &      942 &     1847 &      589 &  -   468\\
 -  2190 &  -  2154 &      684 &      561 &     1369 &      589 &     5941 &     1614\\
    1614 &  -   468 &      302 &  -   492 &  -   412 &  -   468 &     1614 &     1638
\end{array}\right)
\end{equation*}
\normalsize
The smallest eigenvalue of $A$ is $\approx 0.00004$ and the smallest eigenvalue of B is $\approx 0.00005$. Thus, $A$ and $B$ are symmetric positive semidefinite. Now, the proof is completed by the following inequalities. 
\begin{eqnarray*}
0.4446-\;\parbox[c]{1cm}{\begin{tikzpicture}
\filldraw[fill=black] (1,1) circle (0.07cm);
\filldraw[fill=black] (0,1) circle (0.07cm);
\filldraw[fill=black] (0.5,0) circle (0.07cm);
\draw[-to,shorten >=2pt,thick] (0.5,0)--(1,1);
\draw[-to,shorten >=2pt,thick] (0,1)--(0.5,0);
\end{tikzpicture}}
& \geq & \left\llbracket \left(\underline{f^1}\right)^TA\underline{f^1}\right\rrbracket_1\geq 0 \Rightarrow i(\vec P_3)\leq 0.4446\\
0.1104-
\;\parbox[c]{1cm}{\begin{tikzpicture}
\filldraw[fill=black] (1,1) circle (0.07cm);
\filldraw[fill=black] (1,2) circle (0.07cm);
\filldraw[fill=black] (2,1) circle (0.07cm);
\filldraw[fill=black] (2,2) circle (0.07cm);
\draw[-to,shorten >=2pt,thick] (1,2)--(1,1);
\draw[-to,shorten >=2pt,thick] (1,1)--(2,1);
\draw[-to,shorten >=2pt,thick] (2,1)--(2,2);
\draw[-to,shorten >=2pt,thick] (2,2)--(1,2);
\end{tikzpicture}}
& \geq &\left\llbracket \left(\underline{f^1}\right)^TB\underline{f^1}\right\rrbracket_1\geq 0 \Rightarrow i(\vec C_4)\leq 0.1104
\end{eqnarray*}
\end{proof}
\begin{theorem}\label{th:k12}
\begin{equation*}
6-4\sqrt{2} \leq i(\vec K_{1,2})=i(\vec K_{2,1}) \leq 0.4644
\end{equation*}
\end{theorem}
\begin{proof}
At first, we observe that $i(\vec K_{1,2})=i(\vec K_{2,1})$, because each extremal orgraph of $i(\vec K_{1,2})$ can be transformed to an extremal orgraph of $i(\vec K_{2,1})$ by changing the directions of each edge. Thus, in the following we consider only $\vec K_{1,2}$.\newline
The lower bound we will get from a recursive construction, which is a generalisation of the lexicographic product constructions, we had found for the other orgraphs before. Let $s:=\frac{2\sqrt{2}-1}{7}$ and $\vec G$ be an orgraph of infinte order. We seperate $V_{\vec G}$ into three parts $S_1$, $S_2$ and $S_3$, such that $S_1$ are $s$ parts of all vertices of $\vec G$, $S_2$ are $\frac{1-s}{2}$ parts and $S_3$ the remaining $\frac{1-s}{2}$ parts. Now $\vec G$ has an edge from every vertex in $S_1$ to every vertex in $S_2\cup S_3$. Furthermore there is no edge between a vertex from $S_2$ and a vertex from $S_3$. Finally, for each $i\in \left[3\right]$ the vertices of $S_i$ contain a copy of $\vec G$. See figure \ref{fig:k12graph} for an illustration of this definition.
\begin{figure}\label{fig:k12graph}
\begin{center}
\begin{tikzpicture}
\draw[-to,shorten >=31pt,very thick] (5,2)--(8,5);
\draw[-to,shorten >=31pt,very thick] (5,2)--(2,5);
\filldraw[fill=white] (5,2) circle (1cm);
\draw (5,2) circle (1cm);
\draw (5,2) node {$\vec G$};
\draw (5,0.5) node {$S_1$: $s$-parts};
\draw (2,5) circle (1.1cm);
\draw (2,5) node {$\vec G$};
\draw (2,6.6) node {$S_2$: $\frac{1-s}{2}$-parts};
\draw (8,5) circle (1.1cm);
\draw (8,5) node {$\vec G$};
\draw (8,6.6) node {$S_3$: $\frac{1-s}{2}$-parts};
\draw (0,3.5) node {$\vec G$:};
\end{tikzpicture}
\end{center}
\caption{The recursive construction of $\vec G$}
\end{figure}
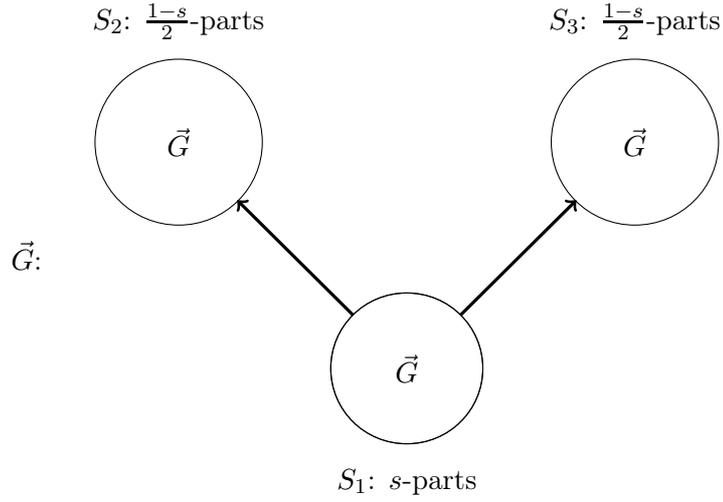
Now we can get the non-edge-density $\bar{\rho}$ of $\vec G$ by
\begin{equation*}
\bar{\rho}=\underbrace{(1-s)\frac{1-s}{2}}_{\text{between}\; S_2\; \text{and} \;S_3 }+\underbrace{\left(s^2+2\left(\frac{1-s}{2}\right)^2\right)\bar{\rho}}_{\text{in each part}}
\end{equation*}
Thus, by rearranging we have
\begin{equation*}
\bar{\rho}=\frac{1-s}{3s+1}.
\end{equation*}
Now we are able to compute the density $d$ of $\vec K_{1,2}$'s in $\vec G$. We get
\begin{equation*}
d=6s\left(\frac{1-s}{2}\right)^2+6s\left(\frac{1-s}{2}\right)^2\bar{\rho}+\left(s^3+2\left(\frac{1-s}{2}\right)^3\right)d,
\end{equation*}
where the first summand is for the $\vec K_{1,2}$'s, where each two vertices are from different parts. The second summand is for the $\vec K_{1,2}$'s, where one vertex is from $S_1$ and the other two vertices are both from $S_2$ or $S_3$. Finally, the last summand is for the $\vec K_{1,2}$'s in each of the three parts. Hence, again by rearranging we get
\begin{equation*}
d=\frac{4(1-s)s}{(1+s)(3s+1)}=6-4\sqrt{2} \Rightarrow i(\vec K_{1,2})\geq 6-4\sqrt{2}.
\end{equation*}
The upper bound we obtain in the same way like in theorem \ref{th:p3c4}. We define $C\in\mathbb{R}^{15\times 15}$ by
\footnotesize
\begin{equation*}
C:=\frac{3}{10^{5}}
\left(
\begin{array}{*{8}{r}}
        15476 &   18421 &   4912 & - 8427 & -  4151 & - 17228 &  -19517 &    \\
        18421 &  124190 &  15060 & -26258 &   86563 &   12365 &  -30370 &    \\
         4912 &   15060 &  14103 & -24269 &   29162 & -  8930 &  -14418 &    \\
       - 8427 & - 26258 & -24269 &  41867 & - 50792 &   14927 &   24879 &    \\
       - 4151 &   86563 &  29162 & -50792 &  234128 &   50298 &  - 1715 &    \\
       -17228 &   12365 & - 8930 &  14927 &   50298 &  109464 &   19869 &    \\
       -19517 & - 30370 & -14418 &  24879 & -  1715 &   19869 &   58166 &  \ldots  \\
       - 2703 &   32824 & -28619 &  48783 & -  9280 &   65336 &   11943 &    \\
       - 1825 & - 23505 & -30259 &  52000 & - 79330 &   12541 &    5415 &    \\
         6567 &   74168 &  22951 & -39939 &  142094 &   35394 &  -10797 &    \\
       - 8635 &   20078 &  23975 & -41342 &   97496 &   12362 &   10837 &    \\
       -39099 & - 16836 &   3666 & - 6625 &   90402 &   62263 &   45590 &    \\
       -18355 & - 34480 & -15446 &  26604 & - 32752 &   16811 &   23634 &    \\
       -14216 &   27335 & -20578 &  34825 &   47003 &   88561 &   30468 &    \\
         5928 &   29429 & -36650 &  62808 & - 59524 &   46894 &     561 &
 \end{array}
\right. 
\end{equation*}
\begin{equation*}
\left.
\begin{array}{*{8}{r}}
  -  2703 & -  1825 &    6567 & - 8635 & - 39099 & -18355 & - 14216 &    5928 \\
    32824 & - 23505 &   74168 &  20078 & - 16836 & -34480 &   27335 &   29429 \\
  - 28619 & - 30259 &   22951 &  23975 &    3666 & -15446 & - 20578 & - 36650 \\
    48783 &   52000 & - 39939 & -41342 & -  6625 &  26604 &   34825 &   62808 \\
  -  9280 & - 79330 &  142094 &  97496 &   90402 & -32752 &   47003 & - 59524 \\
    65336 &   12541 &   35394 &  12362 &   62263 &  16811 &   88561 &   46894 \\
    11943 &    5415 & - 10797 &  10837 &   45590 &  23634 &   30468 &     561 \\
   191285 &   71426 &    3893 & -52488 & - 10271 &  20412 &  113660 &  167521 \\
    71426 &   79987 & - 56787 & -72396 & - 37361 &  28315 &   38299 &  100215 \\
     3893 & - 56787 &  155775 &  64490 &   39105 & -34596 &   34033 & - 25666 \\
  - 52488 & - 72396 &   64490 &  80839 &   70076 & -17545 & -  9052 & - 89649 \\
  - 10271 & - 37361 &   39105 &  70076 &  155515 &  24670 &   44067 & - 59957 \\
    20412 &   28315 & - 34596 & -17545 &   24670 &  32563 &   20238 &   23056 \\
   113660 &   38299 &   34033 & - 9052 &   44067 &  20238 &  158708 &  103921 \\
   167521 &  100215 & - 25666 & -89649 & - 59957 &  23056 &  103921 &  232223
\end{array}\right)
\end{equation*}
 \normalsize
Again, the smallest eigenvalue of $C$ is $\approx 0.000007$. Hence, $C$ is symmetric positive semidefinite. Now the proof of the upper bound is completed by the following inequality.
\begin{equation*}
0.4644-\;\parbox[c]{1cm}{\begin{tikzpicture}
\filldraw[fill=black] (1,1) circle (0.07cm);
\filldraw[fill=black] (0,1) circle (0.07cm);
\filldraw[fill=black] (0.5,0) circle (0.07cm);
\draw[-to,shorten >=2pt,thick] (0.5,0)--(1,1);
\draw[-to,shorten >=2pt,thick] (0.5,0)--(0,1);
\end{tikzpicture}}
\; \geq \; \left\llbracket \left(\underline{f^1}\right)^TC\underline{f^1}\right\rrbracket_1\geq 0 \Rightarrow i(\vec K_{1,2})\leq 0.4644
\end{equation*}
\end{proof}
\subsection{Some remarks}
Most parts of the proofs of the upper bounds in the theorems \ref{th:p3c4} and \ref{th:k12} were done by a computer. At first we decided to work in $\mathcal{O}_5$. Thus, if we take products of two $1$-flags on $3$ vertices, then our calculus works in $\mathcal{O}_5$. After that we used a computer program to calculate the equation which the semidefinte matrix $A$ have to fulfill such that we can prove the associated upper bound. Finally, the determination of the matrices was simply done by a sufficiently close rational approximation to the outcome of a numerical semidefinite-program-solver.  The decision to took the type $1$ came from computer experiments. Suprisingly, we don't get better results, if we take larger types, because in another paper \cite{sperfeld} we got better results with types of higher order. For example we tried the calculation with $9$ types on $3$ vertices and flags of appropriate order such that our calculation works in $\mathcal{O}_5$, but we didn't get better upper bounds in this case. 

\end{document}